\documentclass[12pt,reqno]{amsart}
\usepackage{amssymb,amsmath,amsthm,amsfonts, shadow,fancybox}

\usepackage[usenames,dvipsnames]{color}
\usepackage[colorlinks=true,citecolor=red,linkcolor=blue]{hyperref}
\usepackage[all]{hypcap}

\usepackage{graphics}

\theoremstyle{plain} \newtheorem{thm}{Theorem}
 \newtheorem{prop}[thm]{Proposition}
\newtheorem{lemma}[thm]{Lemma}

\theoremstyle{definition} 
\newtheorem{defn}[thm]{Definition}

\newtheorem{ex}[thm]{Example} 
\theoremstyle{remark} \newtheorem{remark}[thm]{Remark}

\newcommand{\diag}{\Delta}
\newcommand{\R}{\mathbb R}
\DeclareMathOperator{\rec}{CR} 
\newcommand{\nw}{\Omega}
\newcommand{\Z}{\mathbb Z}

\newcommand{\N}{\mathbb N}

\newcommand{\nb}{\dot}
\DeclareMathOperator{\expn}{E} 
\DeclareMathOperator{\shad}{S} 
\DeclareMathOperator{\cl}{cl} 
\DeclareMathOperator{\Int}{int} 
\DeclareMathOperator{\per}{Per}

\begin{document}
\title[Spectral decomposition for top.\ Anosov homeomorphisms]{Spectral decomposition for topologically Anosov homeomorphisms on noncompact and non-metrizable spaces}

\author{Tarun Das}\address{The M. S. University of Baroda, Vadodara, Gujarat, India}\email{tarukd@gmail.com}
\author{Keonhee Lee}\address{Chungnam National University, Daejeon, 305-764, Korea}\email{khlee@cnu.ac.kr}
\author{David Richeson}\address{Dickinson College\\ Carlisle, PA 17013, USA} \email{richesod@dickinson.edu} 
\author{Jim Wiseman} \address{Agnes Scott College \\ Decatur, GA 30030, USA} \email{jwiseman@agnesscott.edu}

\date{\today}

\begin{abstract}
We introduce topological definitions of expansivity, shadowing, and chain recurrence for homeomorphisms. They generalize the usual definitions for metric spaces. We prove various theorems about topologically Anosov homeomorphisms (maps that are expansive and have the shadowing property) on noncompact and non-metrizable spaces that generalize theorems for such homeomorphisms on compact metric spaces. The main result is a generalization of Smale's spectral decomposition theorem to topologically Anosov homeomorphisms on first countable locally compact paracompact Hausdorff spaces.
\end{abstract}

\maketitle

\section{Introduction}

The goal of this paper is to extend the following result, Smale's spectral decomposition theorem applied to Anosov diffeomorphisms of compact manifolds, to more general topological spaces. 

\begin{thm}[\cite{Smale:1967}] \label{thm:smale}
Let $M$ be a compact  manifold and $f:M\to M$ be an Anosov diffeomorphism. Then the non-wandering set $\nw (f)$ can be written as a finite union of disjoint closed invariant sets on which $f$ is topologically transitive. 
\end{thm}

Recall that a diffeomorphism is \emph{Anosov} if it has a hyperbolic structure on the entire manifold and that every Anosov diffeomorphism is expansive and has the shadowing property (see, e.g., \cite[Sect.~9.2]{Robinson:1995}). Theorem \ref{thm:smale} has been extended to homeomorphisms on compact metric spaces:

\begin{thm}[{\cite[{Theorem~11.13}]{A}, \cite[{Theorem~3.4.4}]{Aoki:1994}}]\label{thm:cpctspectral}
Let $X$ be a compact metric space and $f:X\to X$ be an expansive homeomorphism with the shadowing property. Then $\nw (f)$ can be written as a finite union of disjoint closed invariant sets on which $f$ is topologically transitive.
\end{thm}

Yang  extended Theorem \ref{thm:cpctspectral} to noncompact metric spaces (\cite[Theorem~4]{Yang}), with the additional (strong) requirement that the chain recurrent set be compact.  

It is well known that some dynamical properties of homeomorphisms on compact spaces (i.e., properties that are conjugacy invariant), may not be dynamical properties on noncompact spaces. For example, a dynamical system on a noncompact metric space may be expansive or have the shadowing property with respect to one metric, but not with respect to another metric that induces the same topology (see  Examples \ref{ex:expdepmet} and \ref{ex:shdepmet}). Of course, it is preferable to have a theory that is independent of any change of (compatible) metric. In this article we give topological definitions of expansiveness and shadowing that are equivalent to the usual metric definitions for homeomorphisms on compact metric spaces and are  dynamical properties for any metric space. Moreover, these definitions apply to non-metrizable spaces as well. For related work see \cite{Bryant:1960,Das:1998,Hurley:1991,Hurley:1992,Lee:1996,Ombach:1987,Richeson:2007}.

Then we extend the spectral decomposition theorem  to dynamical systems on spaces that are not necessarily metrizable and not necessarily compact. The only concession is that in the noncompact case, the collection of basic sets in the decomposition need not be finite.

\begin{thm}\label{thm:spectral}
Let $X$ be a first countable locally compact paracompact Hausdorff  space and $f:X\to X$ be an expansive homeomorphism with the shadowing property. Then $\nw (f)$ can be written as a union of disjoint closed invariant sets on which $f$ is topologically transitive. If $X$ is compact, then this decomposition is finite.
\end{thm}

\begin{remark}
Every metric space is first countable, paracompact, and Hausdorff (\cite[\S 5]{SteenSeebach}), so the preceding theorem applies to any locally compact metric space.
\end{remark}

The paper is organized as follows. We discuss definitions and preliminaries in the next section.  In Section~\ref{sect:product}, we show that topologically Anosov homeomorphisms have a local product structure, and obtain topological stability as a consequence.  In Section~\ref{sect:theorem} we investigate the properties of the non-wandering and chain recurrent sets, and prove our main result, Theorem~\ref{thm:spectral}.  Finally, we discuss the relationships between the new topological definitions of expansiveness and shadowing and the existing metric definitions.

\section{Preliminaries}

The following are the standard definitions for expansiveness and shadowing, which we will refer to as metric expansiveness and metric shadowing.

\begin{defn}
Let  $(X,d)$ be a metric space and  $f:X \to X$ be a homeomorphism.
\begin{enumerate}
\item $f$ is \emph{metric expansive} if there is an $e>0$ such that for any distinct $x,y\in X$, there exists $n\in\Z$ such that $d(f^n(x), f^n(y))>e$.  The number $e$ is called an \emph{expansive constant}.

\item For $\delta>0$, a \emph{$\delta$-chain} is a sequence $\{x_0,x_1,\dots,x_n\}$ ($n\ge1$) such that $d(f(x_{i-1}), x_i) < \delta$ for $i=1,\ldots,n$.  A \emph{$\delta$-pseudo-orbit} is a bi-infinite $\delta$-chain.

\item $f$ has the \emph{metric shadowing property} if for every $\varepsilon > 0$, there exists a $\delta>0$ such that every $\delta$-pseudo-orbit $\{ x_i \}$ is $\varepsilon$-traced by a point $y$; that is, $d(f^i(y), x_i)<\varepsilon$ for all $i$.

\item A point $x$ is \emph{metric chain-recurrent} if for any $\delta>0$, there is a $\delta$-chain from $x$ to itself.

\end{enumerate}

\end{defn}
These definitions clearly require that $X$ be a metric space.  Furthermore, as we see in the following two examples, if $X$ is noncompact, these properties depend on the choice of metric; a homeomorphism that is metric expansive for $d$ may not be for $d'$, even if $d$ and $d'$ induce the same topology.  Thus on noncompact spaces, these properties are not invariant under topological conjugacy.  In addition, if $f$ has these properties, its iterates may not (\cite[Example~1]{BryantColeman}, \cite[Example~9]{Richeson:2007}).

\begin{ex}\label{ex:expdepmet}
Let $T:\R^{2}\to \R^{2}$ be the linear automorphism induced by the matrix
\begin{equation*}
 \left(
 \begin{array}{rr}
   2  & 0 \\
   0  & \frac{1}{2} \\
 \end{array}
 \right),
\end{equation*}
and consider the stereographic projection $P:S^{2}- \{(0,0,1)\}\to \R^{2}$ defined by $f(x,y,z)= \frac{(x,y)}{(1-z)}$. Then $T$ is not expansive if $\R^{2}$ has the metric induced by $P$, but $T$ is expansive when $\R^{2}$ has the usual metric. Moreover, both metrics induce the same topology on $\R^{2}$.
\end{ex}

\begin{ex}\label{ex:shdepmet}
Let $X\subset \R^2$ be the subset $\bigcup_{n\in\Z} X_n=\{n\} \times [0,2^{-|n|}]$, with the metric inherited from $\R^2$, and define $f:X\to X$ by
$$f(n,y)= \begin{cases} (n+1, 2y) \text{, if $n<0,$} \\ (n+1, \frac12 y) \text{, if $n\ge0$.}
\end{cases}$$
Choose $\delta>0$.  Since $\operatorname{diam}(X_n) < \delta$ for $n$ with 
$|n|>n_0=\lceil\frac1\delta\rceil$, the uniform continuity of $f$ and $f^{-1}$ on the compact set $X_{-n_0}\cup\dots\cup X_{n_0}$ shows that $f$ has the metric shadowing property.  However, $f$ is topologically conjugate to the map on the space $\bigcup_{n\in\Z}\left(\{n\}\times[0,1]\right)$ given by translation on the first coordinate and identity on the second, which clearly does not have the metric shadowing property.
\end{ex}

To address these issues, we make purely topological definitions for these notions.  Our definitions are equivalent to the metric ones in the case that $X$ is a compact metric space; we discuss the relationships among the definitions in general in Section~\ref{sec:relation}.

Recently, the third and fourth authors gave two topological generalizations of positive expansivity (\cite{Richeson:2007}). For a given dynamical system $f:X\to X$, they use the product map $F\equiv f\times f$ on $X\times X$ and neighborhoods of the diagonal $\diag_{X} = \{(x,x): x\in X\}$ to extend the notion of positive expansivity. The idea is that $x$ is close to $y$ in $X$ if and only if the point $(x,y)$ is close to the diagonal $\diag_X$ in $X\times X$.  Thus, instead of requiring that $d(x,y)$ be less than $\varepsilon$, we can require that $(x,y)$ be in a given neighborhood of $\diag_X$. This approach is useful for extending dynamical properties from the compact setting to the noncompact setting. 

We make the standing assumption that all topological spaces are first countable, locally compact, paracompact, and Hausdorff and that all maps are homeomorphisms.  To avoid confusion, we  denote subsets of the product space $X\times X$ by $A$, $M$, $U$, etc., and subsets of the base space $X$ by $\nb A$, $\nb M$, $\nb U$, etc.

Let $U$ be a neighborhood of $\diag_X$, and let $U[x]= \{y\in X:(x,y)\in U\}$ be the cross section of $U$ at $x\in X$. 
For any point $x \in X$ and any neighborhood $\nb G$ of $x$, we can find a neighborhood $U$ of  $\diag_{X}$ such that $U[x]\subset \nb G$. A set $M\subset X\times X$ is \emph{proper} if for any compact subset $S$, the set $M[S]=\bigcup_{x\in S}M[x]$ is compact.  $M$ is \emph{symmetric} if $M$ is equal to its transpose, $M^T=\{(y,x): (x,y)\in M\}$.   Note that if $M$ is a neighborhood of $\diag_X$, then $M\cap M^T$ is a symmetric neighborhood of $\diag_{X}$; thus we can often work with symmetric neighborhoods without loss of generality.

Let 
\begin{align*}
U^{n}= &\{(x,y):\text{there exists } z_{0}=x,z_{1},\ldots,z_{n}=y \in X\\&\text{ such that } (z_{i-1},z_{i}) \in U\text{ for } i=1,\ldots,n\}.
\end{align*} 
The intuition behind this definition is that if we wanted $U$ to be the topological equivalent of having an $(\varepsilon/n)$-ball at each point, then $U^n$ would take the place of an $\varepsilon$-ball at each point. Observe that if $U$ is symmetric, then the order of the elements within each ordered pair does not matter.  Also, for any neighborhood $U$ of $\diag_{X}$ we can find a neighborhood $V$ of $\diag_{X}$ such that $F(V^{2})\subset U$.

\begin{defn}
A homeomorphism $f:X\to X$ is \emph{(topologically) expansive} if there is a closed neighborhood $N$ of $\diag_X$ such that for any distinct $x,y \in X$ there exists $n\in \Z$ such that $F^n(x,y)\not\in N$. Such a neighborhood $N$ is called an \emph{expansive neighborhood} for $f$. Let $\expn(X)$ denote the set of expansive homeomorphisms.
\end{defn}

\begin{lemma}
If $f$ is expansive, then it has a proper expansive neighborhood.
\end{lemma}

\begin{proof}
Let $N$ be an expansive neighborhood for $f$.  Since $X$ is locally compact, each $x\in X$ has an open, relatively compact neighborhood $\nb U_x$.  Since $X$ is paracompact, the open cover $\{\nb U_x\}$ has a closed (and hence compact) locally finite refinement $\{\nb V_\alpha\}$ (\cite[Chapter~5, Theorem~28]{Kelley}). Then $A=N\cap\left(\bigcup_{\alpha}\left(\nb V_{\alpha}\times \nb V_{\alpha}\right)\right)$ is a proper expansive neighborhood.
\end{proof}

Let $D$ and $E$ be neighborhoods of $\diag_{X}$. A \emph{$D$-chain} is a sequence $\{x_0,x_1,\dots,x_n\}$ ($n\ge1$) such that $(f(x_{i-1}), x_i)\in D$ for $i=1,\ldots,n$.  A \emph{$D$-pseudo-orbit} is a bi-infinite $D$-chain. A $D$-pseudo-orbit $\{x_{i}\}$ is $E$-\emph{traced} by a point $y\in X$ if $(f^{i}(y), x_{i})\in E$ for all $i\in \Z$.

\begin{defn} A homeomorphism $f:X\to X$ has the \textit{(topological) shadowing property} if for every neighborhood $E$ of  $\diag_{X}$, we can find a neighborhood $D$ of $\diag_{X}$ such that every $D$-pseudo-orbit is $E$-traced by some point $y\in X$. Let $\shad(X)$ denote the set of homeomorphisms of $X$  with the shadowing property.
\end{defn}

\begin{remark}\label{rem:cpt}
If $(X,d)$ is a compact metric space, then for any neighborhood $U$ of $\diag_{X}$, we  can find $\delta >0$ such that $U_{\delta} = d^{-1}[0,\delta) \subset U$. On the other hand, every $U_{\delta}$  is a neighborhood of $\diag_{X}$. Thus, the above definitions coincide with the usual notions of expansivity and shadowing on compact metric spaces. However, the following example shows that this argument does not hold if $X$ is not compact.
\end{remark}

\begin{ex} 
If $X$ is a noncompact metric space, then a neighborhood $U$ of $\diag_{X}$ may not contain an open set $U_{\delta}$ (defined as above) for any $\delta > 0$. 
For instance, consider the neighborhood $U=\{(x,y)\in \R^{2}: |x-y| < e^{-x^2}\}$ of $\diag_\R$. There is no $\delta > 0$ with $U_{\delta}\subset U$.
\end{ex}

The following properties are easy consequences of the definitions.  As noted earlier, they do not hold for the metric definitions in the noncompact case.

\begin{prop}\label{prop:nice}
\begin{enumerate}
	\item  $f\in \expn(X)$ [resp. $\shad(X)$] if and only if $f^{k} \in \expn(X)$ [resp. $\shad(X)$] for all nonzero $k\in \Z$.
	\item If $h:X\to Y$ is a homeomorphism, then $f\in \expn(X)$ [resp. $\shad(X)$] if and only if $(h\circ f \circ h^{-1})\in \expn(Y)$ [resp. $\shad(Y)$].
\end{enumerate}
\end{prop}

\begin{defn}
A homeomorphism $f:X\to X$ is \emph{topologically Anosov} if $f\in \expn(X)\cap \shad(X)$.
\end{defn}

\begin{lemma}\label{lem:uniqueshadowing}
Let $f$ be topologically Anosov. If $E$ is a neighborhood of $\diag_{X}$ such that $E^{2}$ is an expansive neighborhood for $f$, then there is a neighborhood $D$ of $\diag_{X}$ such that every $D$-pseudo-orbit is $E$-traced by exactly one orbit of $X$. If the $D$-pseudo-orbit is periodic, then so is the orbit that $E$-traces it.
\end{lemma}
\begin{proof}
Let $E$ be a symmetric neighborhood of $\diag_{X}$ such that $E^{2}$ is an expansive neighborhood for $f$. Because $f$ has the shadowing property, there exists a neighborhood $D$ of $\diag_{X}$ such that any $D$-pseudo-orbit is $E$-traced by some point in $X$. We must now show uniqueness. Suppose $\{x_{i}:i\in \Z\}$ is a $D$-pseudo-orbit and that $x,y\in X$ both $E$-trace it. Then for all $i$, $(x_{i},f^{i}(x)),(x_{i},f^{i}(y))\in E$. So $(f^{i}(x),f^{i}(y))\in E^{2}$ for all $i$. Since $E^{2}$ is an expansive neighborhood for $f$, $(x,y)\in\diag_{X}$, and hence $x=y$.

Now suppose $\{y_{k}\}$ is a periodic $D$-pseudo-orbit with period $n$ and let $x$ be the unique $E$-tracing point. Then $\{y_{k+n}\}$ is a periodic $D$-pseudo-orbit and since $f^{n}(x)$ is an $E$-tracing point, it must be the unique one. But $\{y_{k}\}=\{y_{k+n}\}$, so it must be the case that $x=f^{n}(x)$, and hence $x$ is periodic.
\end{proof}

\section{Stable and unstable sets and topological stability}\label{sect:product}

We now introduce notions of stable and unstable sets.

\begin{defn} Let $x\in X$ and $B$ be a neighborhood of $\diag_{X}$.
\begin{enumerate}
\item
The \emph{local stable set of $x$ relative to $B$} is \[W^{s}_{B}(x)= \{y\in X: F^{i}(x,y)\in B, \forall\, i \geq 0\},\]
\item 
the \emph{local unstable set of $x$ relative to $B$} is \[W^{u}_{B}(x)= \{y\in X: F^{i}(x,y)\in B, \forall\, i \leq 0\},\]
\item 
the \emph{stable set} of $x$ is 
\begin{align*}
W^{s}(x)=& \{y\in X: \forall\text{ neighborhood }B\text{ of }\diag_{X},\ \exists\, n \in \N\text{ such that }\\&F^{i}(x,y)\in B, \forall\, i  \geq n \},\text{ and}
\end{align*}
\item 
the \emph{unstable set} of $x$ is 
\begin{align*}
W^{u}(x)=& \{y\in X: \forall\text{ neighborhood }B\text{ of }\diag_{X},\ \exists\, n \in \N\text{ such that }\\&F^{i}(x,y)\in B, \forall\, i  \leq -n \}.
\end{align*}
\end{enumerate}
\end{defn}

A neighborhood $U \subset X\times X$  of the diagonal $\diag_X$ is \emph{wide} if there is a compact set $\nb S\subset X$ such that $U\cup (\nb S\times X)=X\times X$. In other words, $U$ is wide if for any $x$ not in the compact set $\nb S$, the cross section $U[x]$ is $X$.

\begin{lemma}[Continuity Lemma]\label{lem:cont}
A map $f:X \to Y$ is continuous if for any wide neighborhood $U$ of $\diag_{Y}$, there exists a neighborhood $V$ of $\diag_{X}$ such that $F(V) \subset U$.
\end{lemma}

\begin{proof}
Let $x$ be any point in $X$ and let $\nb B\subset Y$ be any neighborhood of $f(x)$; to prove continuity we will find a neighborhood $\nb A$ of $x$ such that $f(\nb A) \subset \nb B$.  Let $U$ be any wide neighborhood of $\diag_{Y}$ such that $U[f(x)]  \subset \nb B$ (for example, $U = (\nb B \times \nb B) \cup ((Y - \cl(\nb B')) \times Y)$, where $\nb B'$ is a neighborhood of $f(x)$ such that $Cl(\nb B') \subset \nb B$).  Then, by hypothesis, there exists a neighborhood $V$ of $\diag_{X}$ such that $F(V) \subset U$.  So $F(x,V[x]) = \{f(x)\} \times f(V[x]) \subset \{f(x)\}\times U[f(x)]$, so $f(V[x]) \subset U[f(x)] \subset \nb B$.  Thus we can take $\nb A= V[x]$.
\end{proof}

\begin{lemma}\label{lem:nbhd}
Let $f\in \expn(X)$ and let $A$ be a proper expansive neighborhood for $f$.  For each $N\in \N$, define $V_N(A) := \{(x,y) \in X\times X : F^n(x,y) \in A \text{ for all } |n|\le N\}$.  Then for any wide neighborhood $U$ of $\diag_X$ there exists $N \in \N$ such that $V_N(A) \subset U$. Conversely, for every $N$ there exists a neighborhood $U$ of $\diag_{X}$ such that $U\subset V_{N}$.
\end{lemma}

\begin{proof}
Suppose there is a wide neighborhood $U$ of $\diag_X$ such that for each $N\in \N$ there exists $(x_{N},y_{N})\in V_{N}(A)\cap (X\times X - U)$. Let $L= \{(x_{N},y_{N}):N \in \N\}$. Because $U$ is a wide neighborhood of $\diag$, there exists a compact set $\nb S\subset X$ such that $U\cup (\nb S\times X)=X\times X$. Then $L \subset (X\times X -U)\cap A \subset \nb S \times A[S]$, which is compact, so $\cl L$ has a limit point, $(p,q)$.  Clearly $p\neq q$. On the other hand, choose a subsequence $(x_{N_{k}},y_{N_{k}}) $ in $L$ converging to $(p,q)$ as $k\to \infty$. Observe that for any integer $i$, $F^{i}(p,q) = \lim_{k\to\infty} F^{i}(x_{N_{k}},y_{N_{k}}) \in A$, since $F^{i}(x_{N_{k}},y_{N_{k}}) \in A$ for $|i| \le N_k$.  Since $F^{i}(p,q)  \in A$ for all $i$ and $A$ is an expansive neighborhood, we must have that $p=q$, which is a contradiction.


For the converse, for each $n$ with $|n|\leq N$ we have a neighborhood $U_{n}$ of $\diag_{X}$ such that $F(U_{n})\subset A$. Take $U = \bigcap_{|n|\leq N} U_{n}$. This completes our proof.
\end{proof}

\begin{lemma}\label{lem:locglob}
If $B$ is a proper expansive neighborhood for $f$ and $x$ is a periodic point, then $W^{\sigma}_{B}(x)\subset W^{\sigma} (x)$, where $\sigma = s,u$.
\end{lemma}
\begin{proof}
We will prove the case $\sigma=s$; the case $\sigma=u$ is similar. For the sake of contradiction, suppose there is a periodic point $x\in X$ and a proper expansive neighborhood $B$ such that $W_{B}^{s}(x)\not\subset W^{s}(x)$. Say that $x$ has period $p$ and the periodic orbit is $\{x_{0},\ldots,x_{p-1}\}$. Because $W_{B}^{s}(x)\not\subset W^{s}(x)$, there is a $y\in W^{s}_{B}(x)$ and a neighborhood $C$ of $\diag_{X}$ such that $F^{n}(x,y)\not\in C$ for infinitely many $n>0$. So, there exist infinitely many points of the form $(x_{j},f^{n_{k}}(y))$ for some $0\le j<p-1$. Because $B[x_{j}]$ is compact, the sequence $\{(x_{j},f^{n_{k}}(y))\}$ has a limit point $(x_{j},y_{0})\in B-\Int(C)$ (without loss of generality, assume that the sequence is convergent). Clearly, $F^{n}(x_{j},y_{0})\in B$ for all $n\ge 0$. For each $i>0$ consider the sequence $\{(f^{-i}(x_{j}),f^{n_{k}-i}(y))\}$. It has a limit  $(x_{j-i},y_{-i})\in \{x_{j-i}\}\times B[x_{j-i}]$ (where the subscript of $x$ is taken mod $p$). Then $f(y_{-i})=y_{1-i}$ for all $i>0$. Thus, $F^{n}(x_{j},y_{0})=(f^{n}(x_{j}),y_{n})\in B$ for all $n<0$. By expansiveness, this implies that $(x_{j},y_{0})\in\diag_{X}$. But $(x_{j},y_{0})\not\in\Int(C)$. So this is a contradiction.
\end{proof}

The results in the following proposition are generalizations of the ones for compact spaces (\cite[Theorem~4.1.1, Lemma 2.4.1(1)]{Aoki:1994}).

\begin{prop}\label{prop:unstabstabsets}
Let $f$ be topologically Anosov. Then we can find neighborhoods $B$ and $D$ of $\diag_{X}$ and a continuous map $t:D\to X$ such that
\begin{enumerate}
\item 
$W^{s}_{B}(x)\cap W^{u}_{B}(y)$ contains at most one point for any  $x, y\in X$,
\item 
$W^{s}_{B}(x)\cap W^{u}_{B}(y)=\{t(x,y)\}$ if $(x,y)\in D$,
\item 
$W^{s}_{B}(x)\cap D[x]= \{y:y=t(x,y),$ where $(x,y)\in D\}$,
\item 
$W^{u}_{B}(x)\cap D[x]= \{y:y=t(y,x),$ where $(x,y)\in D\}$,
\end{enumerate}
\end{prop}

\begin{proof}
Let $A$ be an expansive neighborhood for $f$. Let $B$ be a symmetric neighborhood of $\diag_{X}$ such that $B^{3}\subset A$, and let $E=B\cap F^{-1}(B)$, which is also a symmetric neighborhood of $\diag_{X}$.  Since $f$ has the shadowing property, there exists a neighborhood $D$ of $\diag_{X}$ such that every $D$-pseudo-orbit is $E$-traced by some orbit in $X$. We claim that $E^{2}$ is an expansive neighborhood for $f$, and hence by Lemma \ref{lem:uniqueshadowing} the tracing point is unique. Let $(x,y)\in E^{2}$. Then there is a $z\in X$ such that $(x,z),(z,y)\in E$. So $(f(x),f(z)),(f(z),f(y))\in B$ and hence $(f(x),f(y))=F(x,y)\in B^{2}\subset B^{3}\subset A$. This implies that $F(E^{2})$ is an expansive neighborhood for $f$, and hence so is $E^{2}$.

For each point $(x,y) \in D$, define a $D$-pseudo-orbit $\{x_{i}\}$ in $X$ by
\[
x_i=\begin{cases}
f^i(x) & \text{if }  i \geq 0 \\
f^i(y) & \text{if } i < 0.
\end{cases}
\]
Let $t(x,y)$ denote the unique $E$-tracing point. This defines a map $t:D\rightarrow X$ (we postpone the proof of continuity to the end of the proof). 

Let $(x, y)\in D$. Then $F^{n}(x,t(x,y))\in E\subset B$ for all $n\geq 0$, so $t(x,y)\in W_{B}^{s}(x)$. Likewise, $F^{n}(y,t(x,y))\in E\subset B$ for all $n < 0$ and since $F(E)\subset B$, $F^{-1}(y,t(x,y)) \in E$ implies $(y,t(x,y)) \in B$. So $t(x,y)\in W_{B}^{u}$. In particular,  \[t(x,y)\in W^{s}_{B}(x)\cap W^{u}_{B}(y).\] Furthermore, the expansivity of $f$ does not allow $W^{s}_{B}(x)\cap W^{u}_{B}(y)$ to have more than one point. This proves (1) and (2).

We prove (3); the proof of (4) is similar.  To prove equality we prove containment in both directions. ($\supset$): Let $y=t(x,y)$ and $(x,y)\in D$. Then $y$ $E$-traces $x$ in forward time. So $F^{n}(x,y)\in E\subset B$ for all $n\ge 0$. Thus $y\in W^{s}_{B}(x)\cap D[x]$. ($\subset$): Suppose $y\in W^{s}_{B}(x)\cap D[x]$. Then, $F^{n}(x,y)\in B$ for all $n\ge 0$. Clearly $F^{n}(y,y)\in B$ for all $n<0$, so $y$ $B$-traces the pseudo-orbit $\{x_{i}\}$ (defined as above). By definition, the point $t(x,y)$ $E$-traces this pseudo-orbit, and because $E\subset B$, it also $B$-traces it. So, for all $n\in\Z$, $(f^{n}(t(x,y),x_{n}))\in B$ and $(f^{n}(y),x_{n})\in B$, which implies that $F^{n}(t(x,y),y)\in B^{2}\subset A$. Thus, by expansiveness, $y=t(x,y)$.

We now show that  the tracing map $t$ is continuous. Let $U$ be a wide neighborhood of $\diag_{X}$. By Lemma \ref{lem:cont} we must find a neighborhood $V$ of $\diag_{X}$ such that $T(V)\subset U$, where $T=t\times t$. By Lemma \ref{lem:nbhd}, we can find $N$ such that $V_{N}(A)\subset U$. Define neighborhoods of $\diag_{X}$  \[W_{1}= \bigcap_{n=0}^{N}F^{-n}(B)\text{ and }W_{2}= \bigcap_{n=0}^{N}F^{n}(B).\] Let $g:X^{4}\to X^{4}$ be the homeomorphism given by \[g(x,y,x_{1},y_{1})=(x,x_{1},y,y_{1}),\] and take \[V=g^{-1}(W_{1}\times W_{2})\cap (D\times D).\]  Clearly $V$ is a neighborhood of $\diag_{D}$. Let $(x,y,x_{1},y_{1})\in V$. Because $(x,y),(x_{1},y_{1})\in D$,  $F^{n}(t(x,y),x),F^{n}(x_{1},t(x_{1},y_{1}))\in B$  for $0\le n\le N$, and because $(x,x_{1})\in W_{1}$, $F^{n}(x,x_{1})\in B$ for $0\le n\le N$. So $F^{n}(t(x,y),t(x_{1},y_{1}))\in B^{3}\subset A$ for $0\le n\le N$. Similarly, we can show that $F^{n}(t(x,y),t(x_{1},y_{1}))\in A$ for $-N\le n\le 0$, and hence $T(V)\subset V_{N}(A)\subset U$. Thus $t$ is continuous.
\end{proof}

\begin{defn}
A homeomorphism $f:X\to X$ is called \textit{topologically stable} if for any neighborhood $B$ of  $\diag_{X}$ there exists a neighborhood $D$ of $\diag_{X}$ such that for any homeomorphism $g:X\to X$ satisfying $(f(x),g(x))\in D$, for all $x\in X$ there exists a continuous self map $h$ of $X$ satisfying $(h(x),x)\in B$, for all $x\in X$ and $f\circ h=h\circ g$.
\end{defn}

We close this section by proving the following extension of a main result in \cite{Walters:1978}.

\begin{thm} 
If X is a first countable locally compact paracompact Hausdorff space and $f:X\to X$ is topologically Anosov, then $f$ is topologically stable.
\end{thm}
\begin{proof}
Let $B$ be a neighborhood of $\diag_{X}$. Let $A$ be a proper expansive neighborhood for $f$ and $E$ be a symmetric neighborhood of $\diag_{X}$ such that $E^{3}\subset A\cap B$. By Lemma \ref{lem:uniqueshadowing} there is a neighborhood $D$ of $\diag_{X}$ such that every $D$-pseudo-orbit is uniquely $E$-traced by some point of $X$. Now take a homeomorphism $g:X\to X$ which is $D$-close to $f$. Let $x\in X$. Because $(f(g^{n-1}(x)),g^{n}(x))\in D$ for all $n\in \Z$, $\{g^{n}(x): n\in \Z \}$ is a $D$-pseudo-orbit for $f$. By the shadowing property, this defines a map $h:X\to X$ sending $x$ to the unique point $h(x)$ that $E$-traces the pseudo-orbit. 

Thus for any $x\in X$, we have $(f^{n}(h(x)),g^{n}(x))\in E$ for $n\in \Z$. In particular, $(h(x),x)\in E\subset E^{3}\subset B$ for all  $x\in X $. Furthermore, substituting $g(x)$ for $x$ yields \[(f^{n}(h(g(x))),g^{n}(g(x)))=(f^{n}((h\circ g)(x)),g^{n}(g(x)))\in E\] and \[(f^{n+1}(h(x)),g^{n+1}(x))=(f^{n}((f\circ h)(x)),g^{n}(g(x))) \in E\] for all  $n\in \Z$. This implies that \[(f^{n}((f\circ h)(x)),f^{n}((h\circ g)(x))) \in E^{2}\subset A \]  for all $n\in \Z$, and so $(f\circ h)(x)=(h\circ g)(x)$. 

Now we will use Lemma \ref{lem:cont} to show that $h$ is continuous. Let $U$ be a wide neighborhood of $\diag_{X}$. By Lemma \ref{lem:nbhd} there exists $N\in \N$ such that $V_{N}(A)\subset U$. Let $W=\bigcap_{|k|\leq N}G^{k}(E)$, where $G=g\times g$.  We must show that $h(W) \subset U$. Let  $(x,y)\in W$. If $|n|\leq N$, 
then $(h(g^{n}(x)),g^{n}(x))$, $(g^{n}(x),g^{n}(y))$, and $(h(g^{n}(y)),g^{n}(y))$ are in $E$, and hence  \[(f^{n}(h(x)),f^{n}(h(y)))= (h(g^{n}(x)),h(g^{n}(y)))\in E^{3}\subset A.\] Thus, $(h(x), h(y))\in V_{N}(A)\subset U$.
\end{proof}

\begin{remark} 
In fact, we can prove something slightly stronger than the conclusion of the previous theorem. If the neighborhood $B$ is symmetric and $B^{2}$ is an expansive neighborhood for $f$, then $h$ is unique. Suppose $h'$ is another continuous map satisfying the same properties as $h$. Then for all $n\in \Z$, $(h(g^{n}(x)),g^{n}(x)),(h'(g^{n}(x)),g^{n}(x))\in B$, and hence $(f^{n}(h(x)),f^{n}(h'(x)))=(h(g^{n}(x)),h'(g^{n}(x)))\in B^{2}$. So, $h(x)=h'(x)$.
\end{remark}

\begin{remark} 
If the homeomorphism $g$ in the above theorem is also expansive with $B^{2}$ an expansive neighborhood  and $B$ is symmetric, then $h$ is injective. Indeed, suppose $h(x)=h(y)$ for some $x,y\in X$. Then for any $n\in \Z$,  \[h(g^{n}(x))=f^{n}(h(x))=f^{n}(h(y))=h(g^{n}(y)).\] In particular, $(h(g^{n}(x)),g^{n}(x)),(h(g^{n}(y)),g^{n}(y))=(h(g^{n}(x)),g^{n}(y))\in B$, so $(g^{n}(x),g^{n}(y))\in B^{2}$ for all $n\in \Z$. Thus $x=y$. 
\end{remark}

\section{Chain recurrence and decomposition theorem} \label{sect:theorem}

Chain recurrence is an important notion in dynamical systems defined on metric spaces. In this section we introduce a new, topological generalization of chain recurrence. We then prove our main theorem, Theorem \ref{thm:spectral}. 

\begin{defn}
Let $x,y\in X$. If there is an $A$-chain from $x$ to $y$ and one from $y$ to $x$ for every neighborhood $A$ of $\diag_{X}$, then we write $x\sim y$. The set $\rec(f)= \{x \in X:x\sim x \} $ is called the \textit{chain recurrent set} for $f$. The relation $\sim$ induces an equivalence relation on $\rec(f)$; the equivalence classes are called \textit{chain components} of $f$. 
\end{defn}

\begin{defn}
The \emph{non-wandering set} of $f$ is $\nw (f)=\{x\in X: \ \text{for any open neighborhood } \nb G  \text{ of } x, f^{n}(\nb G)\cap \nb G \neq \emptyset\text{ for some } n>0\}$.
\end{defn}

\begin{prop}\label{prop:nwchrec}
If $f\in S(X)$, then $\nw (f)= \rec(f)$.
\end{prop}

\begin{proof}
It is clear from the definitions that $\nw(f) \subset \rec(f)$, so we must show the opposite inclusion.  Let $x \in \rec(f)$ and $\nb U$ be an open neighborhood of $x$. Let $E$ be a neighborhood of $\diag_{X}$ such that $E[x]\subset \nb U$. By the shadowing property there exists a neighborhood $D$ of $\diag_{X}$ such that every $D$-pseudo-orbit is $E$-traced by some point. Since $x$ is chain recurrent, there exists a $D$-chain $\{x_{0}=x,x_{1},\ldots,x_{n}=x\}$.  We can extend this $D$-chain (in any way we want) to a full $D$-pseudo-orbit. This pseudo-orbit is $E$-traced by a point $y \in X$. This means that $y,f^{n}(y)\in E[x]\subset \nb U$, and so $x \in \nw (f)$.
\end{proof}

We have the following generalization of a result proved by Yang (\cite[Lemma~1(2)]{Yang}); Yang uses the metric definitions of expansivity, shadowing, and chain recurrence.

\begin{prop}\label{prop:denseper}
If $f$ is topologically Anosov, then $\per(f)$ is dense in $\rec(f)$. 
\end{prop}

\begin{proof} Let $x\in \rec(f)$ and $\nb U$ be an open set containing $x$. Let $E$ be a neighborhood of $\diag_{X}$ such that $E[x]\subset\nb U$ and  $E^{2}$ is an expansive neighborhood for $f$. It suffices to show that there is a periodic point in $E[x]$.

By Lemma \ref{lem:uniqueshadowing}, we can find a symmetric neighborhood $D$ of $\diag_{X}$ such that every periodic $D$-pseudo-orbit is $E$-traced by a periodic point. Let $V=D\cap F^{-1}(D)$. By Proposition \ref{prop:nwchrec}, $\rec(f)=\nw(f)$, so $x$ is non-wandering. Thus, there is a $y\in V[x]$ and an $n\in \N$ such that $f^{n}(y)\in V[x]$. Then $(x,y),(x,f^{n}(y))\in V$ and hence $(f(x),f(y)),(f^{n}(y),x)\in D$. This enables us to define a periodic $D$-pseudo-orbit $\{\ldots,x_{0},x_{1},\ldots,x_{n-1},x_{0},x_{1},\ldots\}$ in which $x_{0}=x$, $x_{1}=f(y)$,\ldots, $x_{n-1}=f^{n-1}(y)$.  By Lemma \ref{lem:uniqueshadowing} there is a periodic point $p$ that $E$-traces this pseudo-orbit. So $p\in E[x]$.
\end{proof}

\begin{prop}\label{prop:chrec}
Let $f:X\to X$ be a homeomorphism and $R$ be a chain component of $f$. Then,
\begin{enumerate}
\item $\rec(f)$ is closed in $X$,
\item $R$ is closed in $X$, and
\item $R$ and $\rec(f)$ are $f$-invariant.
\end{enumerate}
\end{prop}

\begin{proof} 
First we will prove that the set \[S= \{(x, y)\in X\times X:\forall\text{ neighborhood }A\text{ of }\diag_{X},\ \exists\ A\text{-chain from }x\text{ to }y\}\] is a closed set in $X\times X$. Let $\{(x_{i},y_{i})\}_{i=0}^{\infty}\subset S$ be a sequence converging to $(x,y)\in X\times X$. Let $A$ be a neighborhood of $\diag_{X}$; we must prove that there is an $A$-chain from $x$ to $y$. 

Choose symmetric neighborhoods $W$ and $U$ of $\diag_{X}$ and $N>0$ such that $W^{3}\subset A$, $U\subset W\cap f^{-1}(W)$, and $(x_{N},x),(y_{N},y)\in U$. Since $(x_{N},y_{N})\in S$ there exists a $W$-chain $\{x_{N}=z_{0},z_{1},\ldots,z_{m}=y_{N}\}$. Suppose $m=1$. Then $(f(x_{N}),y_{N})\in W$. Moreover, because $(x,x_{N})\in U\subset f^{-1}(W)$, $(f(x),f(x_{N}))\in W$. Also, $(y_{N},y)\in U\subset W$. So $(f(x),y)\in W^{3}\subset A$, and hence $\{x,y\}$ is an $A$-chain. Likewise, if $m>1$, we can show that $(f(x),z_{1}),(f(z_{m-1}),y)\in W^{2}\subset A$, and therefore $\{x,z_{1},\ldots,z_{m-1},y\}$ is an $A$-chain. Hence $(x,y)\in S$.

Now, let $S^{T}=\{(x,y)\in X\times X: (y,x)\in S\}$ and $\tilde S=S\cap S^{T}$. Then $\tilde S = \{(x, y)\in X\times X:x \sim y\}$. Since $S$ and $S^{T}$ are closed sets, so is $\tilde S$.

(1) The diagonal $\diag_{X}$ is closed in $X\times X$. Thus, so is  $\rec(f)=\pi_{1}\left(\diag_{X}\cap \tilde S\right)$, where $\pi_{1}:X\times X\to X$ is the projection onto the first coordinate.

(2) Given any $x\in X$, $\tilde S[x]$ is closed, and $\tilde S[x]$ is the chain component of $f$ containing $x$.

(3) Let $x\in R$. We must show that $f(x)\in R$. Let $A$ be a neighborhood of $\diag_X$. Clearly $\{x,f(x)\}$ is an $A$-chain from $x$ to $f(x)$. Thus it remains to find an $A$-chain from $f(x)$ to $x$.  Choose neighborhoods $B$ and $C$ of $\diag_X$ such that $B^2\subset A$ and $C\subset B\cap f^{-1}(B)$.  Since $x$ is chain recurrent, there is a $C$-chain $\{x=x_0, x_1,\ldots,x=x_n\}$ from $x$ to itself.  Then $\{f(x), x_2, x_3,\ldots,x=x_n\}$ is the desired $A$-chain.
\end{proof}

\begin{prop}\label{prop:openclosed}
If $f$ is topologically Anosov and  $R$ is a chain component, then $R$ is both open and closed in $\rec(f)$.
\end{prop}

\begin{proof}
Let $R$ be a chain component of $f$. It follows  from Proposition \ref{prop:chrec} that $R$ is closed in $\rec(f)$, so we must only show that $R$ is open in $\rec(f)$. Let $D$ be the neighborhood of $\diag_{X}$, which we may assume is open and symmetric, given by Proposition \ref{prop:unstabstabsets}. Because $D[R]$ is open, it suffices to show that a point $x\in D[R]\cap \rec(f)$ is in $R$. Let $E$ be a neighborhood of $\diag_{X}$; we will find an $E$-chain from $x$ to $x$ via a point of $R$.

Let $U$ be a neighborhood of $\diag_{X}$ such that $U^{2}\subset E\cap D$. By Proposition \ref{prop:denseper}, there is a periodic point  $p\in U[x]\subset D[x]$. Then there exists $q\in R$ such that $(p,q),(q,p)\in D$; by Proposition \ref{prop:denseper} we may assume $q$ is periodic. By Proposition \ref{prop:unstabstabsets}(2) and Lemma \ref{lem:locglob}, there exist $z_{1},z_{2}\in X$ such that $z_1 \in W^{u}(p)\cap W^{s}(q)$ and $z_2 \in W^{s}(p)\cap W^{u}(q)$. So, there exist $L,M,N,K\in\N$ such that $(f(p),f^{-L}(z_{1}))$, $(q,f^{M}(z_{1}))$, $(f(q),f^{-N}(z_{2}))$, $(p,f^{K}(z_{2}))\in U$. Then \[\{p, f^{-L}(z_{1}), \ldots, f^{M-1}(z_{1}),q, f^{-N}(z_{2}), \ldots, f^{K-1}(z_{2}),p\}.\] is a $U$-chain, and hence \[\{x, f^{-L}(z_{1}), \ldots, f^{M-1}(z_{1}),q, f^{-N}(z_{2}), \ldots, f^{K-1}(z_{2}),x\}.\]  is an $E$-chain.
\end{proof}

\begin{defn}
A map $f:X\to X$ is  \emph{topologically transitive} if for any pair of non-empty open sets $\nb G,\nb H\subset X$, there exists $n\in \N$ such that $f^{n}(\nb G)\cap \nb H\neq \emptyset$.
\end{defn}

We now show that a map that is expansive and has the shadowing property is topologically transitive on each chain component. This is a generalization of a result of Yang's (\cite[Lemma~3]{Yang}) in which he uses the metric definitions of expansivity, shadowing, and chain recurrence.

\begin{prop}\label{prop:chcomp}
If $f$ is topologically Anosov and  $R$ is a chain component, then $f|_{R} : R\to R$ is topologically transitive.
\end{prop}

\begin{proof}
Let $\nb G$ and $\nb H$ be nonempty open sets in $R$. Let $x\in \nb G$ and $y\in \nb H$. Let $E$ be a neighborhood of $\diag_{X}$ such that $x\in E[x]\cap R\subset \nb G$, $y\in E[y]\cap R\subset \nb H$,  $E[x]\cap \rec(f)\subset R$, and $E^{2}$ is an expansive neighborhood for $f$. By Lemma \ref{lem:uniqueshadowing}, there exists a neighborhood $D$ of $\diag_{X}$ (we may also assume $D\subset E$) such that every periodic $D$-pseudo-orbit is $E$-traced by a periodic orbit.  Since $x,y \in R$, there is a $D$-chain from $x$ to itself through $y$, $\{x_{0}=x,\ldots,x_n=y,\ldots,x_{m}=x\}$ in $X$ which we can extend to a periodic $D$-pseudo-orbit, $\{\ldots,x_{m-1},x_{m}=x_{0},\ldots,x_{m}=x_{0},x_{1},\ldots\}$.  By Lemma \ref{lem:uniqueshadowing} there is a periodic point $p\in X$ that $E$-traces this pseudo-orbit.  In fact, $p\in E[x]\cap \rec(f)\subset E[x]\cap R\subset \nb G$ and $f^{n}(p)\subset E[y]\cap R\subset \nb H$. So $f^{n}(p)\in f^{n}(\nb G)\cap \nb H\ne\emptyset$.
\end{proof}

Finally, we have all the ingredients to prove our main theorem.
\begin{proof}[Proof of Theorem \ref{thm:spectral}]
Proposition \ref{prop:nwchrec} tells us that $\nw(f)=\rec(f)$. We know that $\rec(f)$ has a natural decomposition into chain components,  $\{R_{\lambda}\}$, which, by Propositions \ref{prop:chrec}(3) and \ref{prop:openclosed}, are open and closed in $\nw(f)$ and are $f$-invariant. By Proposition \ref{prop:chcomp} $f$ is topologically transitive on each such component. Finally, notice that if $X$ is compact, then so is $\nw(f)$, and because  $\{R_{\lambda}\}$ is a open cover of $\nw(f)$ by disjoint sets, it must be finite.
\end{proof}

\section{Relationships to metric definitions}\label{sec:relation}

In this section with discuss the relationships between our topological definitions and the usual metric definitions for expansiveness, shadowing, and chain recurrence.  First, we give another generalization of chain recurrence, due to Hurley.

\begin{defn}[\cite{Hurley:1992}]
Let $(X,d)$ be a metric space and $f:X \to X$ be a homeomorphism.  Let $\mathcal P$ denote the set of continuous functions from $X$ to $(0,\infty)$.  For $\delta(x)\in\mathcal P$, a \emph{$\delta(x)$-chain} is a sequence $\{x_0,x_1,\dots,x_n\}$ ($n\ge1$) such that $d(f(x_{i-1}), x_i) < \delta(f(x_{i-1}))$ for $i=1,\ldots,n$.  A point $x$ is \emph{strongly chain-recurrent} if for any $\delta \in \mathcal P$, there is a $\delta(x)$-chain from $x$ to itself.
\end{defn}

In the case that we have a metric space, the various properties are related in the following way.

\begin{prop}\label{prop:metricrelations}
Let $(X,d)$ be a metric space and $f:X \to X$ be a homeomorphism.
\begin{enumerate}

\item Metric expansivity implies topological expansivity, but not vice versa.
	
\item Metric shadowing does not imply topological shadowing, nor does topological shadowing imply metric shadowing.
	
\item  Strong chain recurrence and topological chain recurrence are equivalent.  They imply metric chain recurrence, but not vice versa.

\end{enumerate}
If $X$ is compact, then metric and topological expansivity are equivalent, metric and topological shadowing are equivalent, and metric, topological, and strong chain recurrence are equivalent.
\end{prop}

\begin{proof}

(1)  If $e$ is an expansive constant, then $U_e=d^{-1}[0,e)$ is an expansive neighborhood; thus metric expansivity implies topological expansivity.  It is not difficult to see that the homeomorphism in Example \ref{ex:expdepmet} is topologically expansive. Thus, the opposite implication does not hold.

(2) The first half follows from Example \ref{ex:shdepmet}. The given homeomorphism has the shadowing property with respect to one metric, but it does not have the topological shadowing property (this is easiest to see with the second metric).

 Let $f$ be the identity map on the space $X=\{x_{n}\}_{n=1}^{\infty}$, where $x_{n}=\sum_{i=1}^{n}(1/i)$, given the metric inherited from $\R$. Since $X$ has the discrete topology, we can pick a neighborhood $D$ of $\diag_X$ such that $D[x]=\{x\}$ for all $x$.  Then a $D$-pseudo-orbit is an actual $f$-orbit (that is, a fixed point). So the topological shadowing property is trivially satisfied.  On the other hand, given $\delta>0$  and $N>1/\delta$, $\{\ldots,x_{N},x_{N},x_{N},x_{N+1},x_{N+2},\ldots\}$ is a $\delta$-pseudo-orbit, but clearly such a pseudo-orbit is not traced by any orbit. So $f$ does not have the metric shadowing property.

(3)  Topological chain recurrence clearly implies strong chain recurrence.  To prove the opposite implication, we must show that for any open neighborhood $U$ of $\diag_X$, there exists a $\delta \in \mathcal P$ such that $B_{\delta} \subset U$, where $B_{\delta} = \{(x,y) : d(x,y) < \delta(x)\}$.  We may assume that $U[x]\subsetneq X$ for all $x$.  Define a function $h:X\to (0,\infty)$ by $h(x)=d(x,X-U[x])$.  Since $h$ is lower semicontinuous, the function $\delta(x) = \inf\{f(y) + d(x,y): y\in X\}$ is continuous (see \cite[Theorem~67.2]{Ziemer}), and $0 < \delta(x) < h(x)$ for all $x$.  Thus $B_{\delta} \subset B_{h} \subset U$.

It is clear that topological and strong chain recurrence imply metric chain recurrence.  Hurley gives an example (\cite[Example~1]{Hurley:1992}) showing that the opposite implication does not hold.

Finally, the equivalences in the compact case are clear, as in Remark~\ref{rem:cpt}.

\end{proof}

Hurley has shown (\cite{Hurley:1992}) that strong (or topological) chain recurrence is the right notion to extend Conley's fundamental theorem of dynamical systems to noncompact metric spaces.  The second example from (2) is expansive but not metric expansive, and it has the shadowing property but not the metric shadowing property; thus Theorem~\ref{thm:spectral} applies, but a version using the usual metric definitions would not.  These examples, together with Proposition~\ref{prop:nice}, suggest that the topological definitions are a useful way to study the dynamical structure of a homeomorphism on a noncompact space, even if the space is compact and  the usual metric definitions could be applied.

\section{Acknowledgements}
This work was supported by BK 21 Maths vision 2013 Project of Korea. First author would like to thank Professor Keonhee Lee for his warm hospitality and useful discussions during his visit to Chungnam National University, Daejeon in Korea, spring semester 2007. The third and fourth authors would like to thank Professor Tarun Das for the invitation to M.\ S.\ University of Baroda for the 2010 ICM satellite conference on Various Aspects of Dynamical Systems and for his invitation to join this research project.

\nocite{*}
\bibliographystyle{amsplain} 
\bibliography{ChainRecurrence}

\end{document}